\documentclass[journal]{IEEEtran}

\ifCLASSINFOpdf
\else
   \usepackage[dvips]{graphicx}
\fi
\usepackage{url}

\usepackage{graphicx}

\usepackage{amsmath,amsthm}

\usepackage{amssymb,amsfonts,xcolor}

\usepackage{hyperref}

\usepackage[normalem]{ulem}

\usepackage[ruled,vlined,linesnumbered,algosection,algo2e]{algorithm2e}

\theoremstyle{plain}
\newtheorem{thm}{\protect\theoremname}
\theoremstyle{remark}

\theoremstyle{plain}

\newtheorem{defn}{\protect\definitionname}
\newtheorem{prop}[thm]{\protect\Propname}

\newtheorem{remark}{Remark}
\usepackage{comment}
\makeatother

\usepackage{babel}
\providecommand{\claimname}{Claim}
\providecommand{\theoremname}{Theorem}
\providecommand{\lemmaname}{Lemma}
\providecommand{\Propname}{Proposition}
\providecommand{\corollaryname}{Corollary}
\providecommand{\definitionname}{Definition}
\DeclareMathOperator*{\argmin}{arg\,min}
\usepackage{prettyref}

\newrefformat{def}{Def.~\ref{#1}}
\newrefformat{thm}{Thm.~\ref{#1}}
\newrefformat{prop}{Prop.~\ref{#1}}
\newrefformat{lem}{Lemma \ref{#1}}
\newrefformat{cor}{Corollary \ref{#1}}
\newrefformat{sec}{Sect.~\ref{#1}}
\newrefformat{sub}{Subsection \ref{#1}}
\newrefformat{rem}{Remark~\ref{#1}}
\newrefformat{alg}{Alg.~\ref{#1}}
\newrefformat{ex}{Example \ref{#1}}
\newrefformat{app}{Appendix \ref{#1}}
\newrefformat{subfig}{Fig.~\ref{#1}}
\newrefformat{fig}{Fig.~\ref{#1}}

\newcommand{\srf}{\textrm{SRF}} 

\newcommand{\SRF}{\srf}

\newcommand{\mm}{\Lambda} % minimax

\begin{document}

\title{Decimated Prony's method for stable super-resolution}

\author{Rami Katz, Nuha Diab, and Dmitry Batenkov
\thanks{Submitted for review on March 24th, 2023. N.~Diab and D.~Batenkov are supported by the Israel Science Foundation Grant 1793/20 and a collaborative grant from the Volkswagen Foundation.}
\thanks{R.~Katz is with the School of Electrical Engineering, Tel Aviv University, Tel Aviv, Israel (e-mail: ramikatz@mail.tau.ac.il). N.~Diab and D.~Batenkov are with the Department of Applied Mathematics, School of Mathematical Sciences, Tel Aviv University, Tel Aviv, Israel (e-mail: nuhadiab@tauex.tau.ac.il, dbatenkov@tauex.tau.ac.il).}}

\markboth{Preprint}
{Katz \MakeLowercase{\textit{et al.}}: Decimated Prony's method for stable super-resolution}

\maketitle

\begin{abstract}

    We study recovery of amplitudes and nodes of a finite impulse train from  noisy frequency samples. This problem is known as super-resolution under sparsity constraints and has numerous applications. An especially challenging scenario occurs when the separation between Dirac pulses is smaller than the Nyquist-Shannon-Rayleigh limit.  Despite large volumes of research and well-established worst-case recovery bounds, there is currently no known computationally efficient method which achieves these bounds in practice.
    % Prony's method is an algebraic technique which fully recovers the  signal parameters in the absence of measurement noise, by reducing the SR problem to the solution of two linear systems, in conjunction with a root-finding step. In the presence of noise, Prony's method may experience significant loss of accuracy. 
    In this work we combine the well-known Prony's method for exponential fitting with a recently established decimation technique for analyzing the super-resolution problem in the above mentioned regime. We show that our approach attains optimal asymptotic stability in the presence of noise, and has lower computational complexity than the current state of the art methods.
\end{abstract}


\IEEEpeerreviewmaketitle

\section{Introduction}

\begin{comment}
  \IEEEPARstart{T}{he} (noisy) SR problem is to recover the amplitudes $\left\{\alpha_k \right\}_{k=1}^n$ and nodes $\left\{x_k\right\}_{k=1}^n$ of a finite impulse train $f(x) = \sum_{k=1}^n \alpha_k \delta(x-x_k)$ from band-limited and noisy spectral measurements
  \begin{equation}\label{eq:noisy-data}
  g(\omega)=\sum_{k=1}^n \alpha_k e^{ 2\pi\mathrm{j} x_k \omega} + e(\omega),\quad \omega \in [-\Omega, \Omega ],
  \end{equation}
   where $\Omega>0$ and $\left\|e \right\|_{\infty}\leq \epsilon$ for $\epsilon>0$.
  The SR problem belongs to a class of inverse problems with multiple applications, including parametric spectral estimation, direction of arrival estimation, finite rate of innovation sampling, time-of-flight imaging and unlimited sensing \cite{stoica2005, batenkov2013a,bhandari2016,bhandari2022,blu2008}. High-order versions of \eqref{eq:noisy-data} are considered in e.g. \cite{batenkov2022, badeau2006}.
\end{comment}

\IEEEPARstart{V}{arious} problems of signal reconstruction in multiple basic and applied settings can be reduced to recovering the amplitudes $\left\{\alpha_k \right\}_{k=1}^n$ and nodes $\left\{x_k\right\}_{k=1}^n$ of a finite impulse train $f(x) = \sum_{k=1}^n \alpha_k \delta(x-x_k)$ from band-limited and noisy spectral measurements
\begin{equation}\label{eq:noisy-data}
g(\omega)=\sum_{k=1}^n \alpha_k e^{ 2\pi\mathrm{j} x_k \omega} + e(\omega),\quad \omega \in [-\Omega, \Omega ],
\end{equation}
 where $\Omega>0$ and $\left\|e \right\|_{\infty}\leq \epsilon$ for some $\epsilon>0$. Due to its widespread applications, this problem has been studied under various guises including tauberian approximation \cite{defigueiredo1982}, parametric spectrum estimation and direction of arrival \cite{stoica2005,badeau2006}, time-delay estimation \cite{kirsteins1987}, sparse deconvolution \cite{li2000},  super-resolution (SR) \cite{donoho1992a, candes2014}  and finite-rate-of-innovation sampling \cite{vetterli2002,blu2008}. Beyond the theoretical modelling, recent advances have shown \eqref{eq:noisy-data} to be the work-horse for emerging areas such as super-resolution tomography and spectroscopy \cite{seelamantula2014,mulleti2017a}, ultra-fast time-of-flight imaging \cite{bhandari2016} and unlimited sensing \cite{bhandari2022}; in such cases, efficient and robust solutions to \eqref{eq:noisy-data} entail pushing real-world capabilities beyond the possibilities of conventional hardware. 

% \noindent{\bf Motivation. }
Despite the theoretical advances on this topic (works cited above and follow-up literature), there are still fundamental research gaps that arise due to ill-posedness and instability of \eqref{eq:noisy-data} in the presence of noise. In particular, an especially challenging regime occurs when the separation $\Delta$ between two or more nodes is smaller than the Nyquist-Shannon-Rayleigh (NSR) limit $1/\Omega$. Recently, min-max error bounds for SR in the noisy regime were derived in the case when some nodes form a dense cluster \cite{batenkov2020,batenkov2021b,batenkov2022}, establishing the fundamental limits of recovery in the SR problem (cf. \prettyref{sec:towards}). However, a tractable and provably optimal algorithm  has been missing from the literature. 

% \noindent{\bf Contributions. }
In this work we develop a reconstruction algorithm inspired by the decimation  approach \cite{batenkov2013,batenkov2013a} (cf.\cite{batenkov2018,batenkov2020,batenkov2022,batenkov2021b} and also \cite{maravic2005,briani2020}). Our procedure (\prettyref{sec:dp}) relies on sampling $g(\omega)$ at $2n$ equispaced and maximally separated frequencies, followed by solving the SR problem thus obtained by applying the classical Prony's method for exponential fitting \cite{prony1795}, whose accuracy has been recently established in \cite{katz2023}. For success of the approach, care needs to be taken to avoid  node aliasing and collisions. Inspired by \cite{batenkov2021b}, this is achieved by considering sufficiently many decimated sub-problems. The result is a 
tractable algorithm, dubbed the Decimated Prony's method (DPM), which has lower computational complexity than the well-established and frequently used ESPRIT algorithm (see e.g. \cite{li2020a}). We further provide theoretical  and numerical evidence which show that DPM achieves the optimal asymptotic stability and noise tolerance guaranteed in the literature. We believe that our results pave the way to developing robust procedures for optimal solution of problems derived from \eqref{eq:noisy-data}.

\section{Towards optimal algorithms}\label{sec:towards}
 Throughout the paper we consider the number of nodes (resp. amplitudes) $n\in \left\{1,2,\dots \right\}$ in \eqref{eq:noisy-data} to be fixed. We assume that the nodes satisfy $\left\{x_k \right\}_{k=1}^n\subseteq \left[ -\frac{1}{2},\frac{1}{2}\right]$. By rescaling the data \eqref{eq:noisy-data}, this assumption poses no loss of generality (see Section 4 in \cite{batenkov2021b}). Let $\left\{\tilde{x}_k \right\}_{k=1}^n$ and $\left\{\tilde{\alpha}_k \right\}_{k=1}^n$ be the approximated parameters obtained via a reconstruction algorithm using  the data \eqref{eq:noisy-data}. 

\begin{defn}[\cite{batenkov2021b}]
Let $\left\{(\alpha_k,x_k) \right\}_{k=1}^n\subseteq U$. Given $\epsilon>0$, the {\textbf{min-max recovery rates}} are
\begin{align*}
&\mm^{x,j}_{\epsilon, U, \Omega} = \inf_{\mathcal{A}:g\mapsto \left\{\tilde{\alpha}_j,\tilde{x}_j \right\}} \ \sup_{\left\{\alpha_j,x_j \right\}} \ \sup_{ \|e\|_{\infty}\le \epsilon} \left|x_j-\tilde{x_j}\right|,\\
&\mm^{\alpha,j}_{\epsilon, U,\Omega} = \inf_{\mathcal{A}:g\mapsto\left\{\tilde{\alpha}_j,\tilde{x}_j \right\}} \ \sup_{\left\{\alpha_j,x_j \right\}} \sup_{\|e\|_{\infty}\le \epsilon} \left|\alpha_j-\tilde{\alpha}_j\right|.
\end{align*}
Here $U$ is some fixed subset in the parameter space, whereas the infimum is over the set of all reconstruction algorithms $\mathcal{A}$ which employ the data \eqref{eq:noisy-data}.
\end{defn}
The min-max rates define the optimal recovery rates achievable by a reconstruction algorithm, in the presence of measurement noise of magnitude less than $\epsilon$ and when the node and amplitude pairs belong to $U$.

It has been well established that the difficulty of SR is related to the minimal separation between nodes, $\Delta:=\min_{s\neq k}|x_s-x_k|$. A particular case of interest, both theoretically and from an applications perspective, concerns signals whose nodes are densely clustered, i.e. $\Delta \ll \frac{1}{\Omega}$ \cite{batenkov2021,kunis2020a, lee1992, li2021, stoica1995,liu2021a}.

% The theoretical evaluation of the min-max rates is a highly challenging problem. 

% \begin{defn}[Clustered configuration]
% $\left\{x_k \right\}_{k=1}^n$ form a $(\ell,h,T,\tau,\eta)$-configuration, where $0<\tau,\eta \leq 1$, $p\geq 2$ and $0<h \leq T$, if there exists a subset of nodes $\mathcal{X}=\left\{ x_k\right\}_{k=m}^{m+\ell-1}$ such that  $\tau h \leq |x_s-x_k|\leq h, \ \forall x_k,x_s\in \mathcal{X}, \ s\neq k$, whereas $\eta T \leq |x_s-x_k|\leq T, \ \forall x_k\notin \mathcal{X}$ and $x_s$ arbitrary with $s\neq k$.
% \end{defn}

\begin{defn}
We call $\{x_k\}_{k=1}^n$ a \emph{clustered configuration} if there is a partition of the nodes into clusters such that the distances between the nodes in each cluster are on the order of $\Delta$, while the inter-cluster distances are on the order of $1/\Omega$ ({\textbf{well-separated}} clusters). \cite{batenkov2020,batenkov2021b}.
\end{defn}

\begin{thm}[\cite{batenkov2021b}]\label{thm:minmax}
Let the {\textbf{super-resolution factor}} satisfy $\srf:={1\over{\Omega\Delta}} \gtrapprox  1$, and $\{|\alpha_k|\}_{k=1}^n$ be (uniformly) bounded. Let there exist a single cluster $\mathcal{X}$ of size $1<\ell < n$, whereas all other (singleton) clusters are well-separated, 
and $\epsilon \lessapprox (\Omega\Delta)^{2\ell-1}$. Then
\begin{align*}
&\mm^{x,j}_{\epsilon, U, \Omega} \asymp
\begin{cases}
\SRF^{2\ell-2}\frac{\epsilon}{\Omega} & x_j \in \mathcal{X}, \\
{\epsilon\over\Omega} & x_j \notin \mathcal{X},
\end{cases} \\
&\mm^{\alpha,j}_{\epsilon, U,\Omega} \asymp
\begin{cases}
\SRF^{2\ell-1}  \epsilon & x_j \in \mathcal{X}, \\
\epsilon & x_j \notin \mathcal{X}.
\end{cases}
\end{align*}
Here, $\lessapprox,\gtrapprox, \asymp$ denote asymptotic inequalities/equivalence up to multiplying constants independent of $\Omega, \Delta$, $\epsilon$, while  $U$ is the set of signals satisfying the assumptions above.
\end{thm}
The min-max bounds establish fundamental recovery limits in any application modeled by \eqref{eq:noisy-data} (e.g. \cite{batenkov2019}).  Related results are known in the signal processing literature for Gaussian noise \cite{lee1992,stoica1995,shahram2005}. Despite a plethora of methods to solve this problem, up to date a tractable algorithm which achieves the min-max rates is missing from the literature. In particular, the frequently used ESPRIT algorithm is sub-optimal, both in terms of the bounds and the threshold SNR \cite{li2020a}.

While the proof of Theorem \ref{thm:minmax} is non-constructive, it motivates a design of an algorithm achieving the optimal rates in practice.
Let $\mathcal{J}:=\left[\frac{1}{2}\frac{\Omega}{2n-1},\frac{\Omega}{2n-1}\right]$. For a {\textbf{\mbox{decimation parameter}}} $\lambda\in\mathcal{J}$, the measurements $\{g(\lambda k)\}_{k=0}^{2n-1}$ yield the problem \eqref{eq:noisy-data} with $\{e^{2\pi\mathrm{j}x_j}\}_{j=1}^n$ replaced by $\{e^{2\pi\mathrm{j}\lambda x_j}\}_{j=1}^n$. By \cite[Prop.~5.8]{batenkov2021b} there exists an interval $\mathcal{I}\subset\mathcal{J}$ of length $|\mathcal{I}| \geq c\Omega$ such that every $\lambda\in \mathcal{I}$ satisfies $|e^{2\pi \mathrm{j} \lambda x_i} - e^{2\pi \mathrm{j} \lambda x_j}| \geq n^{-2}$, whenever $x_i,x_j$ belong to different clusters ({\bf{collision avoidance}}). By \cite[Prop.~5.12]{batenkov2021b}, for a collision-avoiding $\lambda$, the condition number of the (theoretical) solution map $\mathcal{P}_n$ which inverts \eqref{eq:noisy-data} in the case $e\equiv0$ and $\omega\in \left\{k\lambda \right\}_{k=0}^{2n-1}$
% $\mathcal{P}_n: \{g(k)\}_{k=0}^{2n-1} \mapsto \{\alpha_j,e^{2\pi\mathrm{j}x_j}\}_{j=1}^n$
matches the min-max rates. Set $\{\tilde{\alpha}_j,e^{2\pi\mathrm{j} \tilde{y}_{\lambda,j}}\}_{j=1}^n=\mathcal{P}_n\left\{g(\lambda k)\right\}_{k=0}^{2n-1}$. Let the set of all {\bf{aliased solutions}} corresponding to $\{\tilde{y}_{\lambda,j}\}_{j=1}^n$ be 
\begin{equation}\label{eq:Xlambda}
  X_{\lambda} :=  \bigcup_{j=1}^{n} \bigg\{\left(\lambda,t\right):\ t=\frac{\tilde{y}_{\lambda,j} + m}{\lambda}, m \in \mathbb{Z}, \left| t \right| \leq \frac{1}{2} \bigg\} , 
\end{equation} 
where the aliasing follows by periodicity of $y \mapsto e^{2\pi \mathrm{j}y}$.
The arguments above imply that $X_{\lambda}$ contains at least one element $(\lambda,t)$ with $t\approx x_j$ for each $j=1,\dots,n$ (call it {\bf Property P*}). Thus, to obtain a constructive procedure for recovery, we propose the following general approach: 
\begin{enumerate}
  \itemsep0em
  \item Find a collision-avoiding $\lambda\in\mathcal{J}$;
  \item Compute $\{\tilde{\alpha}_j,e^{2\pi\mathrm{j} \tilde{y}_{\lambda,j}}\}_{j=1}^n\approx\mathcal{P}_n\left\{g(\lambda k)\right\}_{k=0}^{2n-1}$ \emph{with optimal stability/accuracy};
  \item Find $\{(\lambda,\tilde{x}_j)\}_{j=1}^n \subset X_{\lambda}$ s.t. $\tilde{x}_j \approx x_j$ ({{\textbf{dealiasing}}}).  
\end{enumerate}

In this work we tackle steps 1 and 3. For step 2, we propose to use the classical Prony's method \cite{prony1795} (see \prettyref{alg:classical-prony}), which provides an exact solution to the problem in the noiseless regime. The use of Prony's method is justified by our recent results in \cite{katz2023} which prove its optimality in the regime $\Delta \ll 1$ and $\Omega$ fixed (corresponding to $\SRF\gg 1$ in \prettyref{thm:minmax}).

\begin{thm}[\cite{katz2023}]\label{thm:prony-optimality}
    Suppose $\ell_*$ is the largest cluster size, and each $x_j$ belongs to a cluster of size $\ell_j$.
    For $\epsilon\lessapprox \Delta^{2\ell_*-1}$, the output of \prettyref{alg:classical-prony} satisfies: $|x_j-\tilde{x}_j|\lessapprox \Delta^{2-2\ell_j}\epsilon$ for the nodes, and $|\alpha_j-\tilde{\alpha}_j| \lessapprox \Delta^{t_j}\epsilon$ for the amplitudes, where $t_j=1-2\ell_j$ if $\ell_j>1$, with $t_j=0$ if $\ell_j=1$.
\end{thm}

Define the node/amplitude {\bf{error amplification factors}} $$\mathcal{K}_{x,j}:=\epsilon^{-1}\Omega|x_j-\tilde{x}_j|, \quad  \mathcal{K}_{\alpha,j}:=\epsilon^{-1}|\alpha_j-\tilde{\alpha}_j|.$$ Fixing $\Omega=2n-1$ for $n\in \left\{ 3,4,5\right\}$ and cluster sizes $\ell\in\{2,3,5\}$, Fig. \ref{fig:vanilla} shows that, indeed, both $\mathcal{K}_{x,j}$ and $\mathcal{K}_{\alpha,j}$ computed by Alg. \ref{alg:classical-prony} scale as the min-max rates above.

\begin{remark}\label{rem:successful-recovery}
In all numerical tests in this paper, we follow \cite{batenkov2021b} and consider $x_k$ to be {\textbf{successfully recovered}} if the error $|x_k-\tilde{x}_k|$ is smaller than one third of the distance between $x_k$ and its nearest neighbor.
\end{remark}

%  Fixing $\Omega=2n-1$, numerical tests of Prony's method for the case $n\in \left\{ 3,5\right\}$ with a cluster of size $\ell=2$ in the regime $\Delta \ll 1$ (corresponding to $\SRF\gg 1$ in \prettyref{thm:minmax}), in the presence of measurement noise are given in \prettyref{fig:vanilla}. The tests provide evidence that \textbf{Prony's method asymptotically achieves the optimal min-max rates} required for step 2.
  
\begin{algorithm2e}[hbt]
  \SetKwInOut{Notation}{Notation} \SetKwInOut{Input}{Input} \SetKwInOut{Output}{Output}
 
  \Input{Sequence $\{\tilde{m}_k\equiv  g(k)\}_{k=0}^{2n-1}$ }
  \Output{Estimates $\{ \tilde{x}_k,\ \tilde{\alpha}_k\}_{k=1}^n$}
  \Notation{$\operatorname{col}\left\{y_k \right\}_{k=1}^m = [y_1,\dots,y_m]^{\top}\in \mathbb{C}^m$.}
    Construct $\tilde{H}_n = \left(\tilde{m}_{i+j} \right)_{0\leq i,j\leq n-1}$
    % \begin{equation*}
    %   \tilde{H}_n = \left(\tilde{m}_{i+j} \right)_{i=0,\dots,n-1}^{j=0,\dots,n-1}
    % \end{equation*}
 \\
    %Assuming $\operatorname{det}(\tilde{H}_n) \neq 0$, 
    Solve the linear least squares problem %$\tilde{H}_n\cdot \operatorname{col}\left\{q_k \right\}_{k=0}^{n-1} = -\operatorname{col}\left\{\tilde{m}_k \right\}_{k=n}^{2n-1}$
    \[\operatorname{col}\left\{q_k \right\}_{k=0}^{n-1} = \underset{\boldsymbol{q} \in \mathbb{C}^n}{\argmin} \big\|\tilde{H}_n\boldsymbol{q} + \operatorname{col}\left\{\tilde{m}_k \right\}_{k=n}^{2n-1}\big\|_2\]
    % \begin{equation*}
    %     \tilde{H}_n\cdot \operatorname{col}\left\{q_k \right\}_{k=0}^{n-1} = -\operatorname{col}\left\{\tilde{m}_k \right\}_{k=n}^{2n-1}
    % \end{equation*}
\\
    Compute $\{ \tilde{z}_k\}_{k=1}^n$ as the roots of the (perturbed) Prony polynomial $ q(z) := z^n+\sum_{j=0}^{n-1} q_{j}z^j$.
\\
    Recover $\left\{\tilde{x}_k \right\}_{k=1}^n$ from $\tilde{z}_k$ via $\tilde{x}_k = \frac{\operatorname{Arg}(\tilde{z}_k)}{2\pi}$.
\\
	 Construct $\tilde{V} = \left( \tilde{z}_k^i\right)_{i=0,\dots,n-1}^{k=1,\dots,n}$ and solve
    %$\left( \tilde{z}_k^i\right)_{i=0,\dots,n-1}^{k=1,\dots,n}\cdot \operatorname{col}\left\{\tilde{\alpha}_k \right\}_{k=1}^n = \operatorname{col}\left\{\tilde{m}_k \right\}_{k=0}^{n-1}$
    $\operatorname{col}\left\{\tilde{\alpha}_k \right\}_{k=1}^n = \underset{\boldsymbol{\alpha} \in \mathbb{C}^n}{\argmin}\big\|\tilde{V}\boldsymbol{\alpha} - \operatorname{col}\left\{\tilde{m}_k \right\}_{k=0}^{n-1}\big\|_2$
% 	\begin{equation*}%\label{eq:vand-system}
% 	    \left( \tilde{z}_k^i\right)_{i=0,\dots,n-1}^{k=1,\dots,n}\cdot \operatorname{col}\left\{\tilde{\alpha}_k \right\}_{k=1}^n = \operatorname{col}\left\{\tilde{m}_k \right\}_{k=0}^{n-1}
% 	\end{equation*}
 \caption{The Classical Prony method}\label{alg:classical-prony}
	\Return the estimated parameters $\left\{\tilde{x}_k,\tilde{\alpha}_k\right\}_{k=1}^n$
\end{algorithm2e}

\begin{figure}[htb]

\begin{minipage}[b]{1.0\linewidth}
    \centering
    \centerline{\includegraphics[width=8.5cm]{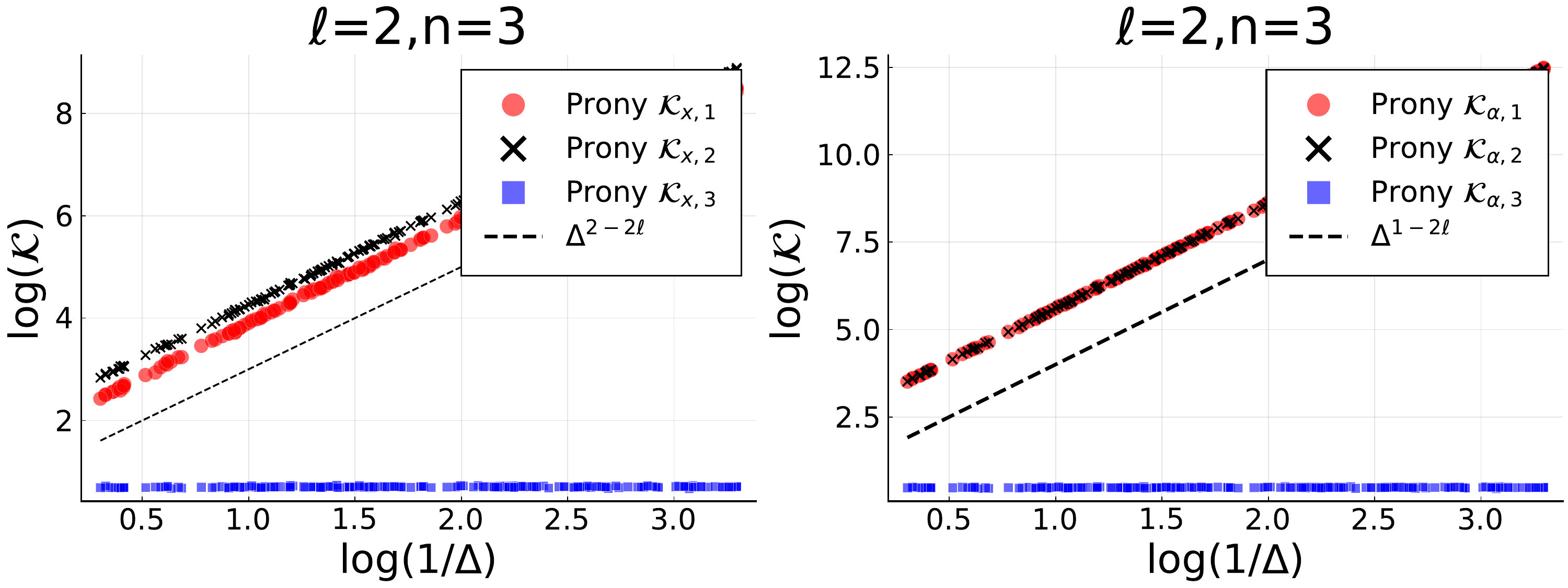}}
  %  \vspace{2.0cm}
    \centerline{(a) Error amplification factors for $\ell=2,n=3$}\medskip
  \end{minipage}

  \begin{minipage}[b]{.48\linewidth}
    \centering
    \centerline{\includegraphics[width=4.0cm]{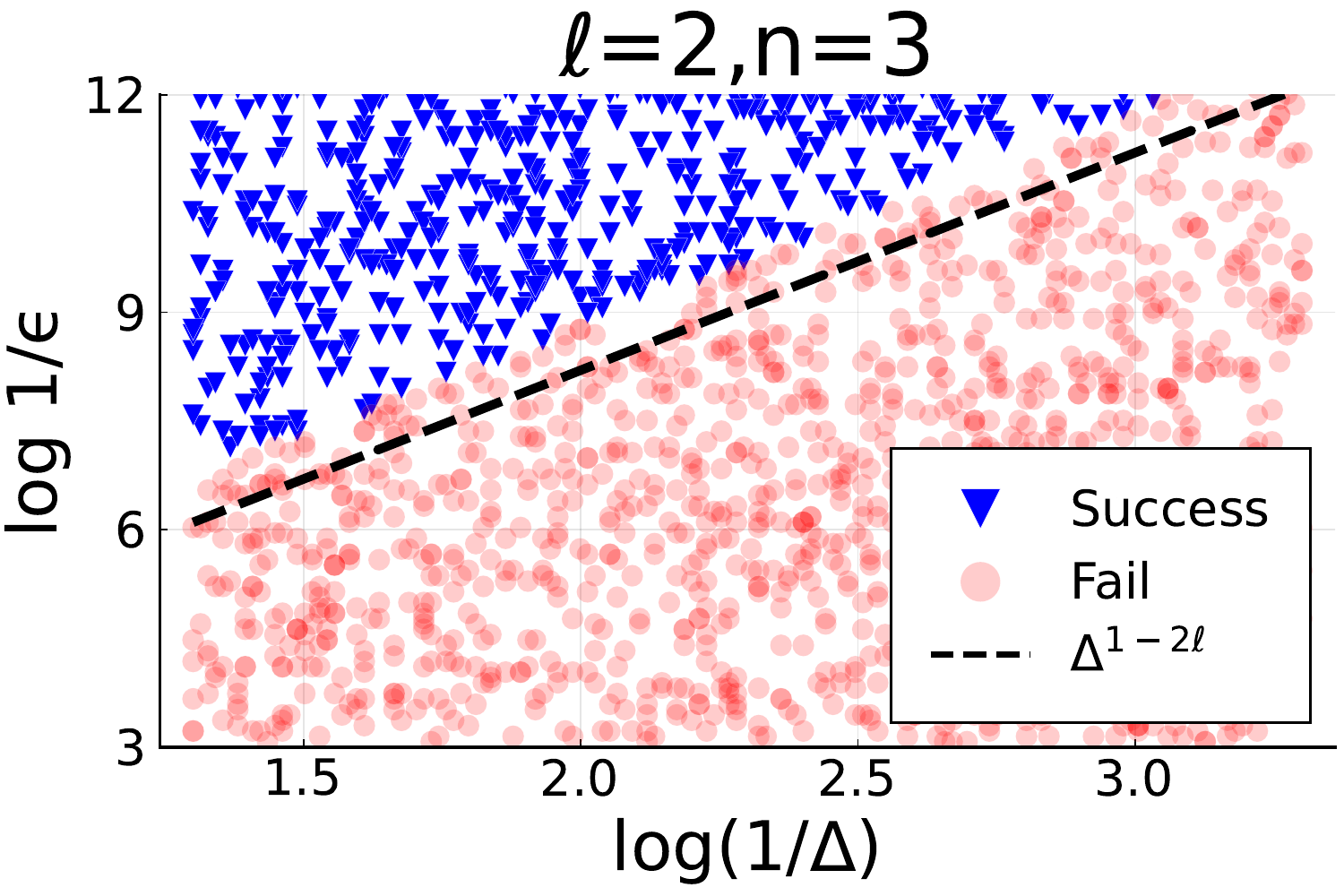}}
  %  \vspace{1.5cm}
    \centerline{(b) Recovery SNR threshold}\medskip
  \end{minipage}
  \hfill
  \begin{minipage}[b]{0.48\linewidth}
    \centering
    \centerline{\includegraphics[width=4.0cm]{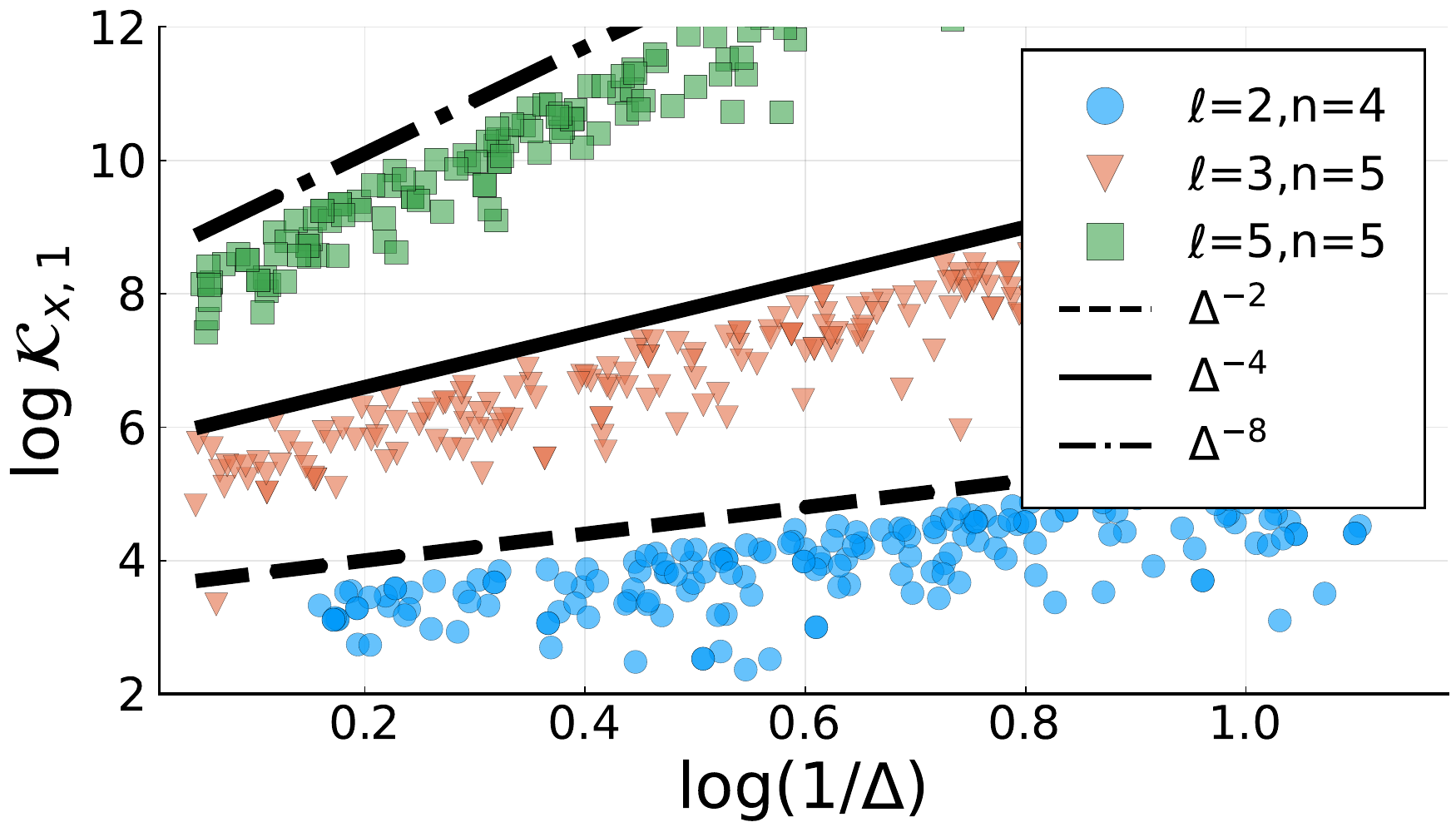}}
  %  \vspace{1.5cm}
    \centerline{(c) $\mathcal{K}_{x,1}$ for various $\ell,n$}\medskip
  \end{minipage}

  \caption{\small Classical Prony method - asymptotic optimality. (a) For cluster nodes, $\left\{\mathcal{K}_{x,j} \right\}$ (left) scale like  $\Delta^{2-2\ell}$, while $\left\{ \mathcal{K}_{\alpha,j}\right\}$ (right) scale like $\Delta^{1-2\ell}$. For the non-cluster node $j=3$, $\left\{\mathcal{K}_{x,j} \right\}$ and $\left\{\mathcal{K}_{\alpha,j} \right\}$ are lower bounded by a constant. (b) Noise threshold for recovery of cluster nodes  scales like $\Delta^{2\ell-1}$. (c) $\mathcal{K}_{x,j}\asymp \Delta^{2-2\ell}$ for cluster nodes holds for different $\ell,n$. The number of tests is (a) 200, (b) 2000, (c) 200 for each $\ell,n$.}
  \label{fig:vanilla}
  \end{figure}

\section{Decimated Prony's method}\label{sec:dp}

Here we develop the Decimated Prony's Method (DPM) (Alg. \ref{alg:DPM}). To find a collision-avoiding $\lambda$, we consider $X_{\lambda}$ as in \eqref{eq:Xlambda} for each $\lambda\in G$ where $G=\operatorname{linspace}(\mathcal{J},N_{\lambda})$ is the uniform grid of size $N_{\lambda}\in\mathbb{N}$ (the choice of $N_{\lambda}$ is motivated in Remark \ref{rem:NbNlambda}. Cf. Fig. \ref{fig:good_lam} for a numerical justification of this approach) [steps 1--5]. Next, we compute the histogram of $\{x: (\lambda,x)\in \bigcup_{\lambda\in G} X_{\lambda}\}$ with $N_b$ bins and find the $n$ bins $\{B_j\}_{j=1}^n$ with largest counts [step 6]. Following the criterion for successful node recovery (see Remark \ref{rem:successful-recovery}), $N_b$ is set to $3\Delta^{-1}$. Based on \cite{batenkov2021b} we expect the set
\begin{equation}\label{eq:lambda}
  \Lambda:=\bigcap_{k=1}^{n} \{\lambda: (\lambda,x) \in \bigcup_{\lambda\in G} X_{\lambda} \wedge x\in B_k\}
\end{equation}
to contain only collision-avoiding $\lambda$'s. In particular, if $\lambda$ is not collision-avoiding, {\bf Property P*} will not be satisfied since at least two nodes will be ill-conditioned. Furthermore, the proof of \cite[Prop.~5.17]{batenkov2021b} suggests that if $\lambda_1\neq\lambda_2$ are collision-avoiding, then, with high probability, $(\lambda_1,t_1)\in X_{\lambda_1}$ and $(\lambda_2,t_2)\in X_{\lambda_2}$ with $t_1\approx t_2$ implies $t_1\approx t_2\approx x_j$ for some $j=1,\dots,n$. Thus, provided $\Lambda\neq\emptyset$ (otherwise the algorithm fails), we choose $\lambda^*=\max \Lambda$ to obtain maximal in-cluster separaton of $\{e^{2\pi\mathrm{j} \lambda^* x_j}\}_{j=1}^n$ [step 7] and recover the corresponding $\{\tilde{x}_j\}_{j=1}^n$ by choosing $\tilde{x}_j \in B_j$ s.t. $(\lambda^*,\tilde{x}_j)\in X_{\lambda^*}$ [step 8]. Finally, the amplitude approximations are found by solving a Vandermonde system [step 9].

We are able to prove (see the Appendix) correctness of DPM in a special case of single cluster (i.e. $\ell=n$).

\begin{thm}\label{thm:correctness-single-cluster}
    In the notations of \prettyref{thm:minmax}, suppose $\ell=n$. Under a further technical assumption (to be elaborated in the proof), and with the choice of $N_{\lambda}=O(\Omega)$ and $N_b=3\Delta^{-1}$, \prettyref{alg:DPM} attains the bounds of \prettyref{thm:minmax}.
\end{thm}

\begin{remark}\label{rem:NbNlambda}
     We conjecture that the choices $N_{\lambda}\gtrapprox\Omega$ and $N_b\approx \Delta^{-1}$ are sufficient to ensure correctness of the algorithm in the general case. We leave the rigorous justification of this claim to future work. As a guideline, increasing $N_{\lambda}$ is expected to improve the  robustness of the DPM.
\end{remark}

Next, we analyze the time complexity of DPM by addressing each step of \prettyref{alg:DPM}. Below we use the notation $O=O_n$ (recall that $n$ is considered fixed). The classical Prony's method has complexity $O(1)$, since it depends only on $n$. For every $\lambda \in G$ we apply \prettyref{alg:classical-prony} with the samples $\left\{g(\lambda k) \right\}$ and compute $X_{\lambda}$ in \eqref{eq:Xlambda}, which costs  $O(N_{\lambda}+\lambda N_{\lambda})$. Computing the histogram with $N_b$ bins for data of size $\sum_{i=1}^{N_{\lambda}}n\lambda_i$ costs $O(N_{\lambda}\Omega + N_b)$. Finding the bins $\left\{B_k\right\}_{k=1}^n$ with $n$ largest counts costs $O(N_b)$. Computing $\Lambda$ in \eqref{eq:lambda} and  $\lambda^* = \max \Lambda$, together with finding the estimated $\{\tilde{x}_k\}_{k=1}^n$ costs  $O(N_{\lambda})$. Finally, solving an $n$-order Vandermonde system costs $O(1)$ \cite{bjork1970}. Thus, the \textit{total complexity} is $O(N_{\lambda}\Omega + N_b)$.  As mentioned in Remark \ref{rem:NbNlambda}, $N_{\lambda} \gtrapprox \Omega$ and $N_b \gtrapprox \Delta^{-1}$ are expected to be sufficient for correctness of DPM. This gives overall complexity $O(\Omega^2)+O(\Delta^{-1})$. For small $\Omega$ the dominating factor is $\Delta^{-1}$, in which case we may take $N_{\lambda}=O(\srf)$ to have maximal robustness. Otherwise, for large $\Omega$ the dominating factor is $O(\Omega^2)$. For comparison, the widely used ESPRIT method \cite{stoica2005} has a time complexity of order $O(\Omega^3)$: it includes three SVD decompositions and several matrix multiplications of order $O(\Omega) \times O(\Omega)$.
\begin{figure}[htb]
  \includegraphics[width=0.49\linewidth]{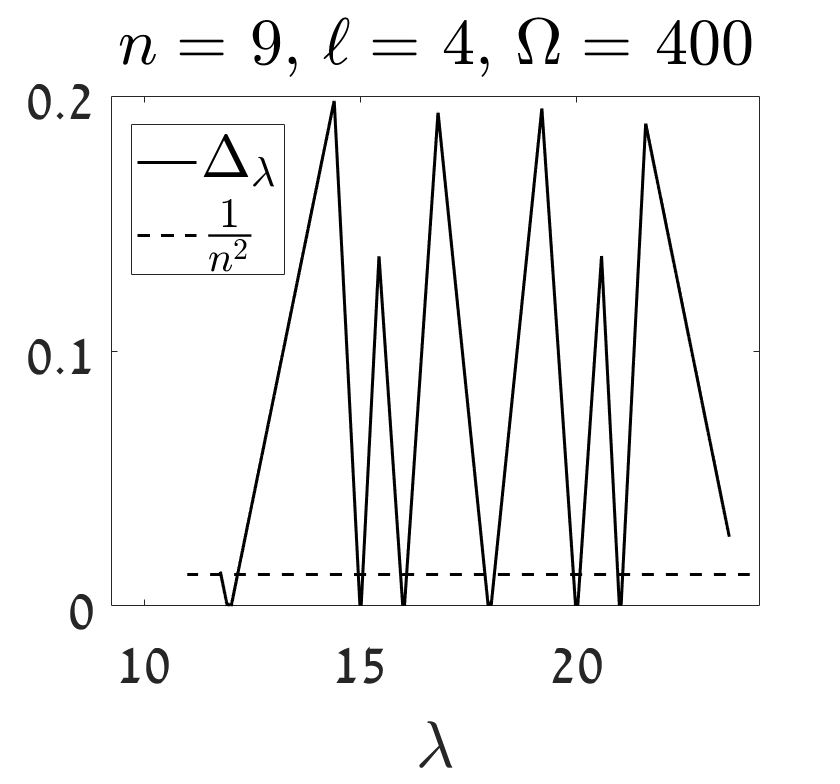} \hfill
  \includegraphics[width=0.49\linewidth]{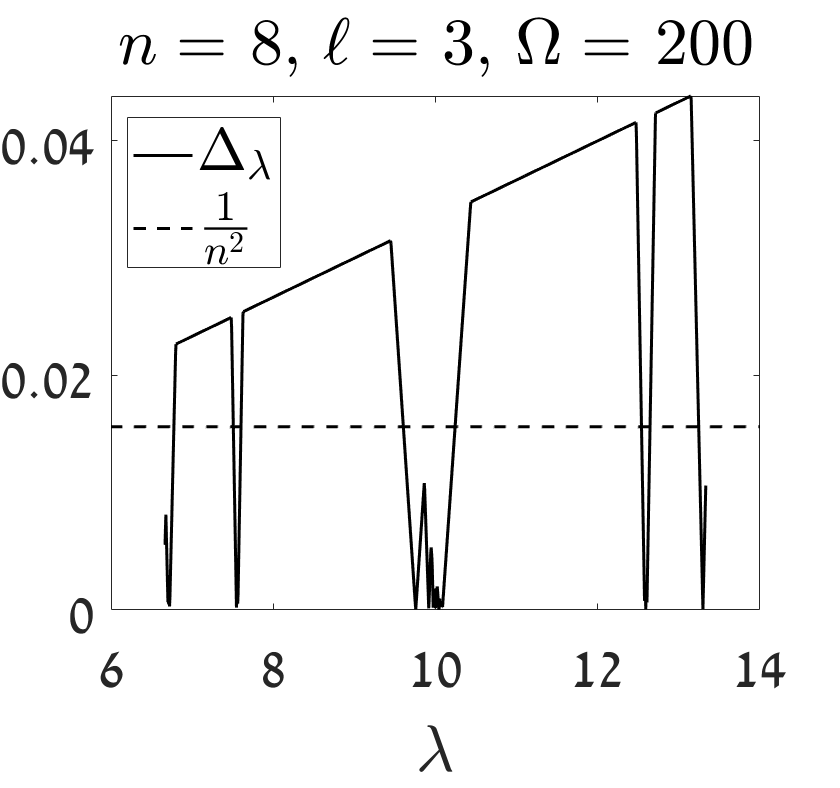}
  \caption{\small  The separation $\Delta_{\lambda} := \min_{s \neq k}\left|\operatorname{arg}\left(e^{2\pi \mathrm{j}\lambda(x_k-x_s)}\right)\right|$ as a function of $\lambda$, for the single (left) and multi-cluster (right) cases. We see that most $\lambda\in G$ satisfy  $\Delta_{\lambda} > \frac{1}{n^2}$, i.e. are collision-avoiding. Here we set $N_{\lambda}=\Omega$.}
  \label{fig:good_lam}
\end{figure}

\begin{algorithm2e}[hbt]
\caption{Decimated Prony Method}\label{alg:DPM}
\SetKwInOut{Input}{Input} \SetKwInOut{Output}{Output} \SetKwInOut{Data}{Data}
\SetKwFunction{Prony}{Prony}
\SetKwFunction{ArgMax}{ArgMax}
\Data{$N_{\lambda}, n, \Omega, N_b$}
\Input{$g(\omega)$ as in \eqref{eq:noisy-data}}
\Output{Estimates $\{ \tilde{x}_k,\ \tilde{\alpha}_k\}_{k=1}^n$}

 \For{$\lambda \in G  := \operatorname{linspace}(\mathcal{J},N_{\lambda})$}{
  $\tilde{m}^{(\lambda)}:= \left\{\tilde{m}^{(\lambda)}_{k}=g(\lambda k) \right\}_{k=0}^{2n-1}$ \\
  $\{e^{2\pi\mathrm{j} \tilde{y}_{\lambda,k}},\tilde{\alpha}_{\lambda,k}\} \longleftarrow \operatorname{Prony}\left(\tilde{m}^{(\lambda)} \right)$
  \\
  Compute $X_{\lambda}$ as in \eqref{eq:Xlambda}
 }
 
 $X \longleftarrow \bigcup_{\lambda \in G} X_{\lambda}$
 \\
  Compute $H$ - Histogram of $\{x: (\lambda,x)\in X\}$ with $N_b$ bins. Set
 $\left\{B_k \right\}_{k=1}^n = $ \ArgMax{$H,n$} 
 \\
 Compute $\Lambda$ as in \eqref{eq:lambda} and $\lambda^* \xleftarrow{} \max\{\lambda: \lambda \in \Lambda\}$
 \\
  $\{\tilde{x}_j\}_{j=1}^n \longleftarrow \{x: (\lambda^*,x) \in X_{\lambda^*} \wedge x\in B_j\}$
%  $\{\tilde{x}_j\}_{j=1}^n \longleftarrow P_x\big(H(B_1,\dots,B_n)\bigcap X_{\lambda^*}\big) {\color{purple}\textbf{$\{y_{\lambda^*,j}\}$?}}$
 \\
 Construct $\tilde{V} = \left( e^{2\pi\mathrm{j}\tilde{x}_j k \lambda^*}\right)_{k=0,\dots,n-1}^{j=1,\dots,n}$ and solve {\small
    $\operatorname{col}\left\{\tilde{\alpha}_k \right\}_{k=1}^n = \underset{\boldsymbol{\alpha} \in \mathbb{C}^n}{\argmin}\biggl\|\tilde{V}\boldsymbol{\alpha} - \operatorname{col}\left\{\tilde{m}_k^{(\lambda^*)} \right\}_{k=0}^{n-1}\biggr\|_2$}

%  Solve $\tilde{V}_{\lambda^*}\cdot {\tilde{\alpha}}=\tilde{m}^{(\lambda^*)}$ where $\tilde{V}_{\lambda^*} = \big[e^{i\lambda^* kx_j}\big]_{k=0,\dots,n-1}^{j=1,\dots,n}$\;
\KwRet{the estimates $\{ \tilde{x}_k,\ \tilde{\alpha}_k\}_{k=1}^n$}
\end{algorithm2e}

We perform reconstruction tests of a signal with random complex amplitudes and measurement noise in a single-cluster configuration  with $\ell<n$, where $\epsilon, \Omega, \Delta$ are chosen uniformly at random. The results appear in Fig. \ref{fig:dp-numerics}. We also investigate the noise threshold  $\epsilon \lessapprox \srf^{1-2\ell}$ for successful recovery (see Remark \ref{rem:successful-recovery}), and compare to \prettyref{thm:minmax} by recording the success/failure result of each experiment. These results provide numerical validation of the optimality of DPM both in terms of the SNR threshold and the attained estimation accuracy.

Finally, we compare the performance of DPM with the ESPRIT method. Fixing $\ell=2,n=3$, $\Delta=10^{-2.8},\Omega=10^{2.5}$ and running 50 tests for each of 10 values of $\epsilon$ between $10^{-3.5}$ and $10^{-2}$, the mean absolute error in recovering the cluster node is comparable between the two methods. However, DPM with $N_{\lambda}=50$ runs about 7 times faster (Fig. \ref{fig:dp-numerics}(c)).

%{\color{blue} Code for reproducing the figures is available at TBD.}

% \vspace{3cm}
%  In particular, \cite[Proposition 5.8]{batenkov2021b} shows existence of $\lambda\in I$ that separates cluster nodes by $O(\Omega \Delta)$ and the non-cluster nodes by at least  $O(\frac{1}{n^2})$. This is validated in \prettyref{fig:good_lam}, where we compute  
 
%  In \prettyref{fig:dp-numerics} we perform numerical tests of reconstructing a signal with random complex amplitudes and measurement noise in a single-clustered configuration. We choose $\epsilon, \Omega, \Delta$ uniformly at random. As in \cite{batenkov2021b}, we measure the error amplification factors of the nodes and amplitudes (denoted $\mathcal{K}_{x}{\color{purple}:= \frac{|x_j-\tilde{x}_j|\cdot \Omega}{\epsilon}}, \mathcal{K}_{\alpha}{\color{purple}:= \frac{|\alpha_j-\tilde{\alpha}_j|}{\epsilon}}$, respectively). We further investigate the noise threshold  $\epsilon \lessapprox \srf^{1-2\ell}$ (corresponding to the SNR threshold $1/\epsilon \gtrapprox \srf^{2\ell-1}$) for successful recovery, as predicted by Theorem \ref{thm:minmax} by recording the success/failure result of each one of the experiments. Our results provide numerical validation of the optimality of DPM both in terms of the SNR threshold and the attained estimation accuracy.

\section{Discussion}
Future research avenues include a rigorous proof of the algorithm's correctness in the general case and improving its robustness, building upon the theory developed in \cite{batenkov2021b,katz2023}.  We believe our method can be extended to higher dimensions, along the lines of recent works such as \cite{diederichs2022}.

% To validate the proposed algorithm, we show numerically that most of the $\lambda$'s attain good decimation properties. By the definition in \cite[Proposition 5.8]{batenkov2020}, a \textit{good decimation parameter} $\lambda$ separates the cluster nodes by $O(\Omega \Delta)$ and the non-cluster nodes at least by $O(\frac{1}{n^2})$ after a blow-up by $\lambda$. We compute $\Delta_{\lambda} := \min_{s \neq k}\big|Arg(\frac{z_s^{\lambda}}{z_k^{\lambda}})\big|$ as a function of $\lambda$ in the single and multi-cluster case (\prettyref{fig:good_lam}). 
\begin{figure}[htb]
\begin{minipage}[b]{1.0\linewidth}
  \centering
  \centerline{\includegraphics[width=8.5cm]{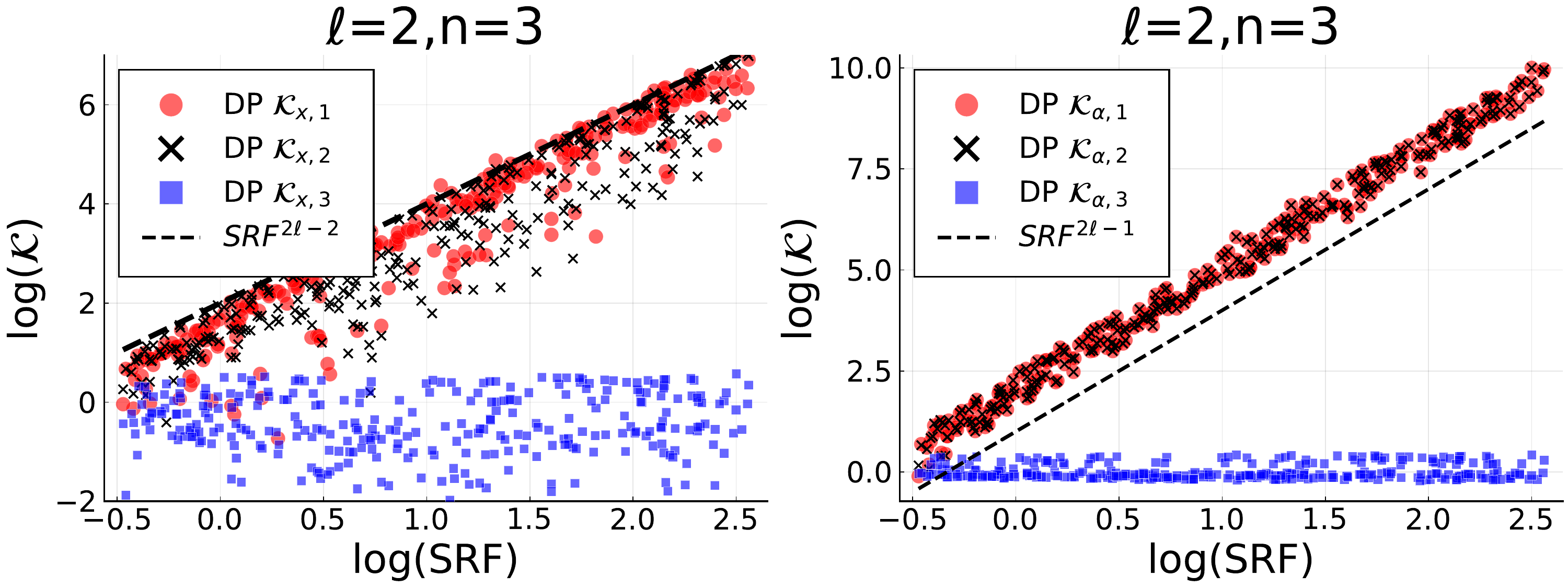}}
%  \vspace{2.0cm}
  \centerline{(a) Error amplification factors for $\ell=2,n=3$.}\medskip
\end{minipage}

\begin{minipage}[b]{.48\linewidth}
  \centering
  \centerline{\includegraphics[width=4.0cm]{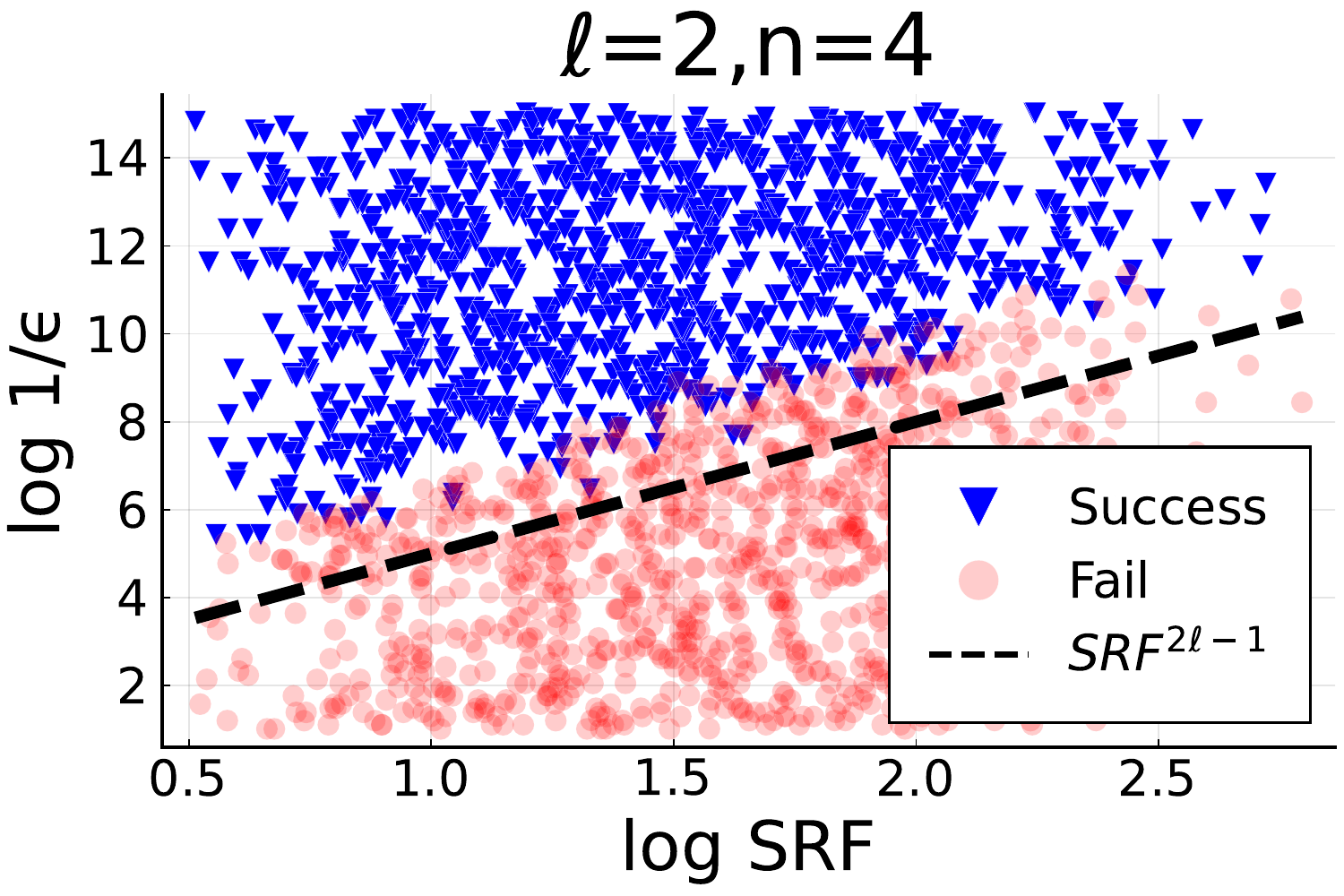}}
%  \vspace{1.5cm}
  \centerline{(b) Recovery SNR threshold.}\medskip
\end{minipage}
\hfill
\begin{minipage}[b]{0.48\linewidth}
  \centering
  \centerline{\includegraphics[width=4.0cm]{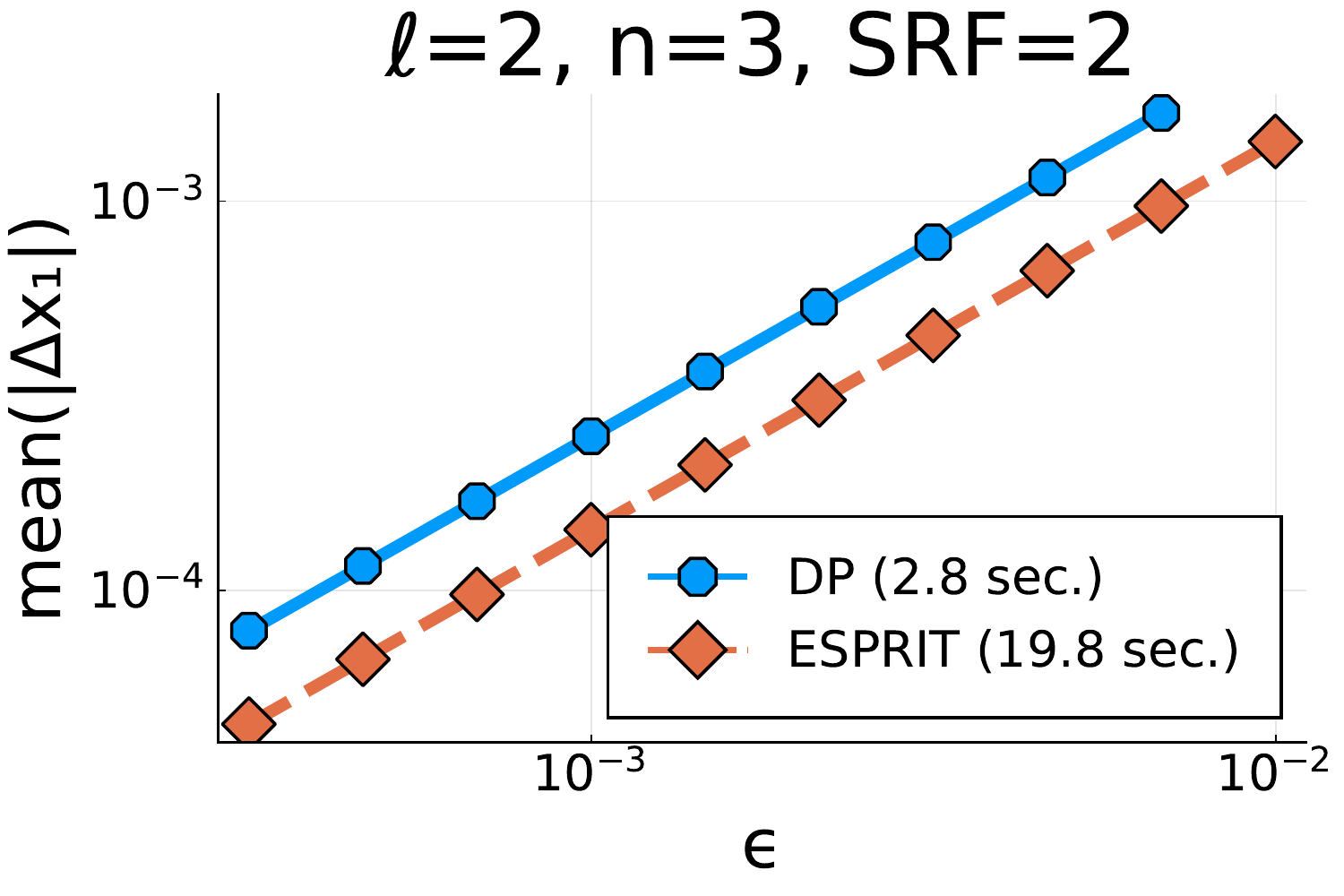}}
%  \vspace{1.5cm}
  \centerline{(c) DPM vs. ESPRIT.}\medskip
\end{minipage}

\caption{\small DPM - asymptotic optimality. (a) For cluster nodes, $\left\{\mathcal{K}_{x,j} \right\}$ (left) scale like ${\srf}^{2\ell-2}$, while the $\left\{\mathcal{K}_{\alpha,j} \right\}$ (right) scale like ${\srf}^{2\ell-1}$ ($\srf=(\Omega\Delta)^{-1}$). For the non-cluster node $j=3$, both $\left\{\mathcal{K}_{x,j} \right\}$ and $\left\{\mathcal{K}_{\alpha,j} \right\}$ are lower bounded by a constant. Here $N_{\lambda}=10$ and number of tests=$300$. (b) Noise threshold for recovery of cluster nodes scales with  ${\srf}^{2\ell-1}$. Here $N_{\lambda}=50$ and number of tests=$2000$. (c) Comparison of accuracy and runtime with the ESPRIT method. While the mean absolute errors (50 tests for each $\epsilon$) scale the same and are comparable, the DPM runs 7 times faster ($N_{\lambda}=50$). }
\label{fig:dp-numerics}
\end{figure}

% \IncMargin{1em}

% \DecMargin{1em}

\appendix

Here we prove \prettyref{thm:correctness-single-cluster}. Recall $\mathcal{J}=\left[\frac{1}{2}\frac{\Omega}{2n-1},\frac{\Omega}{2n-1}\right]$. Note that in our case (a single cluster) each $\lambda\in\mathcal{J}$ is collision-avoiding. Hence, we only need to show that the bins $\{B_j\}_{j=1}^n$ contain $N_{\lambda}$ valid node approximations. We shall use the following technical result.

\begin{prop}\label{prop:pf-main-technical-prop}
    Let $h,\eta>0$ be arbitrary. There exist constants $K_1(n),K_2(n)$ such that for each $h<|c|\leq \eta/6$ and each $\Omega\in\left[{K_1\over\eta},{K_2\over h}\right]$ there exists an interval $I\subset\mathcal{J}$ of length $|I|=\eta^{-1}$ satisfying
    $$
    \forall \lambda\in I,\;\forall k\in\mathbb{Z}\qquad \left|c-\frac{k}{\lambda}\right| > h.
    $$
\end{prop}
\begin{proof}
    This is just a simplified version of \cite[Prop.~F.3]{batenkov2021b}. In the proof, the interval $I_1$ in case 1 may be replaced by any $I\subset\mathcal{J}$ of length $\eta^{-1}$ and appropriate adjustment of the constants $K_1,K_2$; in case 2 the interval $I$ can be taken as $I_5$ itself.
\end{proof}

\begin{proof}[Proof of \prettyref{thm:correctness-single-cluster}]
    Without loss of generality we assume that $x_j\in[-\tilde{h},\tilde{h}]$ for each $j=1,\dots,n$ where $\tilde{h}=\tau \Delta$. Now suppose $\lambda_0\in G$ and introduce the auxiliary parameter $a<\min(1/6,\tau)$. Letting $\epsilon  \lessapprox \Delta^{2n-1}$ be small enough and employing \prettyref{thm:prony-optimality}, the set $X_{\lambda_0}$ in \eqref{eq:Xlambda} can be guaranteed to have form $X_{\lambda_0} = \bigcup_{m\in R(\lambda_0)} X_{\lambda_0,m}$, where 
    $$
    R(\lambda_0)=\biggl\{-\biggl\lfloor \frac{\lambda_0}{2} \biggr\rfloor,\dots, \biggl\lfloor \frac{\lambda_0}{2} \biggr\rfloor \biggr\} 
    $$
   where
    $$
    X_{\lambda_0,m}=\biggl\{ \tilde{x}_{j,\lambda_0} + \frac{m}{\lambda_0} \biggr\}_{j=1}^n,\;|\tilde{x}_{j,\lambda_0}-x_j|\leq a \Delta.
    $$

    Now recall step 6 in \prettyref{alg:DPM}.  We make the following

    \noindent{\bfseries\underline{Genericity assumption:}} The distance from $x_j$ to its closest bin edge is at least $a\Delta$.

    For each $\lambda'\in G$, $|\tilde{x}_{j,\lambda_0}-\tilde{x}_{j,\lambda'}| \leq 2a\Delta < \Delta/3$. Therefore, both $\tilde{x}_{j,\lambda_0}, \tilde{x}_{j,\lambda'}$ must belong to the same bin. In particular, we are guaranteed the existence of $n$ bins containing at least $N_{\lambda}$ elements each.

    Fix $m\in R(\lambda_0)\setminus\{0\}$. The choice of $a$ guarantees $X_{\lambda_0,m} \subset c+\left[-2\tilde{h},2\tilde{h}\right]$ where $c:=\frac{m}{\lambda_0}$. Put $h=6\tilde{h}$. Since $|c| \geqslant \lambda_0^{-1} \geqslant \frac{2n-1}{\Omega}$, the condition $|c|>h$ holds whenever $\srf > \textrm{const}$. Thus we can apply \prettyref{prop:pf-main-technical-prop} with  $\eta=3$ and $h,c$ as above, and conclude that there exists $I_m \subset \mathcal{J}$ of length $|I_m|=1/3$ s.t.
    \begin{equation}\label{eq:far-away}
    \forall \lambda\in I_m,\;\forall k\in\mathbb{Z}:\qquad \biggl| \frac{m}{\lambda_0} - \frac{k}{\lambda} \biggr| > 6\tilde{h}.
    \end{equation}
    By choosing $N_{\lambda}\geq C_1 \Omega$ for sufficiently large $C_1$, we can ensure that there exists $\lambda_m \in I_m \bigcap G$. Therefore, for each $k\in\mathbb{Z}$, \eqref{eq:far-away} implies that
    \begin{align*}
    X_{\lambda_m,k} &\subset \left[-1/2,1/2\right] \setminus \biggl(\frac{m}{\lambda_0} + \left[-2\tilde{h},2\tilde{h}\right]\biggr).
    \end{align*}
    Since $X_{\lambda_0,m} \subset\frac{m}{\lambda_0} + \left[-2\tilde{h},2\tilde{h}\right]$,
    we conclude that $\tilde{x}_{j,\lambda_0}+\frac{m}{\lambda_0}$ and $\tilde{x}_{j,\lambda_m} + \frac{k}{\lambda_m}$ cannot belong to the same bin. In particular, the bin containing $\tilde{x}_{j,\lambda_0}+\frac{m}{\lambda_0}$ contains at most $N_{\lambda}-1$ elements.
    
    Since $\lambda_0$ and $m$ were arbitrary, we have shown that the bins containing $\{\tilde{x}_{j,\lambda}: \lambda \in G\}$ have counts at least $N_{\lambda}$ for each $j=1,\dots,n$, while all other bins have strictly smaller counts. Thus, the former bins will be selected when thresholding the histogram.
\end{proof}

\begin{remark}
    The genericity assumption is a technical and not an essential restriction. \prettyref{alg:DPM} can be easily modified to account for the case that all valid approximations to a node $x_j$ belong to two neighboring bins.
\end{remark}

\bibliographystyle{IEEEbib}
\bibliography{refs}

\end{document}